\documentclass{amsart}
\usepackage{tikz-cd}
\usepackage{xcolor,colortbl}
\usepackage{latexsym,amssymb}
\usepackage[english]{babel}
\usepackage{bm}
\usepackage{amsfonts, amsthm}
\usepackage{amsmath}
\usepackage{mathrsfs}
\usepackage{multirow}
\usepackage{bm}
\usepackage{pgf,tikz}
\usepackage{skak}
\usepackage{hyperref}

\newtheorem{prop}{Proposition}[section]
\newtheorem{lem}[prop]{Lemma}
\newtheorem{cor}[prop]{Corollary}
\newtheorem{thm}[prop]{Theorem}

\theoremstyle{definition}
\newtheorem{rem}[prop]{Remark}
\newtheorem{defi}[prop]{Definition}
\newtheorem{ex}[prop]{Example}
\newtheorem*{KN}{Crazy Knight's Tour Problem}
\newtheorem{conjecture}[prop]{Conjecture}

\newtheorem*{conjectureabs}{Conjecture}

\numberwithin{equation}{section}

\newcommand{\probname}{Crazy Knight's Tour Problem}

\usepackage[pass]{geometry}
\usepackage{tabularx}
\usepackage{tikz}

\newcommand\tikzmark[2]{%
\tikz[remember picture,baseline] \node[inner sep=2pt,outer sep=0] (#1){#2};%
}

\newcommand\link[2]{%
\begin{tikzpicture}[remember picture, overlay, >=stealth, shift={(0,0)}]
  \draw[->] (#1) to (#2);
\end{tikzpicture}%
}

\newcommand{\WidestEntry}{$\bullet$}%
\newcommand{\SetToWidest}[1]{\makebox[\widthof{\WidestEntry}]{$#1$}}%

\usepackage{mathtools}

\begin{document}

\title[A conjecture on the Crazy Knight's Tour Problem]{A conjecture on the Crazy Knight's Tour Problem}

\author[L. Mella]{Lorenzo Mella}
\address{}
\email{lorenzo.mella@unibs.it}

\author[A. Pasotti]{Anita Pasotti}
\address{DICATAM - Sez. Matematica, Universit\`a degli Studi di Brescia, Via Branze 43, I-23123 Brescia, Italy}
\email{anita.pasotti@unibs.it}


\begin{abstract}
Let $A$ be an $m\times n$ toroidal array containing filled and empty cells. Fix an orientation $R=(r_1,\dots,r_m)$
of each row and an orientation $C=(c_1,\dots,c_n)$ of each column of $A$. Given an initial filled cell $(i_1,j_1)$ consider the list
$ L_{R,C}=((i_1,j_1),(i_2,j_2),\ldots,(i_k,j_k),$ $(i_{k+1},j_{k+1}),\ldots)$
where $j_{k+1}$ is the column index of the filled cell $(i_k,j_{k+1})$ of the row $R_{i_k}$ next to $(i_k,j_k)$ in the orientation $r_{i_k}$,
and where $i_{k+1}$ is the row index of the filled cell of the column $C_{j_{k+1}}$ next to $(i_k,j_{k+1})$ in the orientation $c_{j_{k+1}}$. The problem is the following.
\begin{KN}
Do there exist $R$ and $C$ such that the list $L_{R,C}$ covers all the filled cells of $A$?
\end{KN}
This problem was introduced in [S. Costa, M. Dalai and A. Pasotti, \textit{A tour problem on a toroidal board}, Austral. J. Combin. \textbf{76} (2020), 183--207] to construct new biembeddings of graphs on surfaces starting from an Heffter array.  

Here we provide solution to the Crazy Knight's Tour Problem for infinite classes of \emph{cyclically $k$-diagonal} square arrays, namely square arrays whose filled cells are exactly those of $k$ consecutive diagonals. These new constructions together with some known results induce us to propose the following.
\begin{conjectureabs}
Let $A$ be a cyclically $k$-diagonal square array of order $n$.  Then there exists a solution to the Crazy Knight's Tour Problem on $A$ if and only if $n$ and $k$ are odd integers with $n\geq k \geq3$.
\end{conjectureabs}
\end{abstract}

\maketitle

\section{Introduction}
Tour problems of chess pieces over chessboards have always been an interesting topic in Discrete Mathematics, that have been widely studied over the years due to their vast variety of applications, like Hamiltonian cycles and graph domination, and due to the fact that many combinatorial problems can be translated in this setting, like graph labelings and embeddings, see \cite{W}.

Here, we consider a tour problem introduced in \cite{CDP}, called \textit{Crazy Knight's Tour Problem}, that arises from a particular class of combinatorial arrays, called Heffter arrays, firstly introduced by Archdeacon in \cite{A}, and then widely studied in successive works. A solution to this tour problem allows the construction of $2$-face colorable embeddings of complete graphs on orientable surfaces, see \cite{CDP} for details. As shown below, the tour problem can be defined more in general on the filled cells  of any array, while the content of each cell is not relevant for the problem.  For this reason, in this work we are not going to report notions and results on Heffter arrays, and we address the interested reader to the survey \cite{PD23}.

In what follows, we consider a chessboard that might have some holes, that we represent as an array where some cells are allowed to be empty, i.e., a partially filled array (the empty cells of the array are colored in gray). Moreover, to avoid trivial cases, we assume that every row and every column of an array contains at least one filled cell. Finally, unless explicitely stated, every array that we consider is toroidal, hence if we have an $m \times n$ array the row and column  indexes of its filled cells are written with their residue class modulo $m$ and modulo $n$, respectively. 

Given two positive integers $a$ and $b$ with $a\leq b$, we denote by $[a,b]$ the set of integers $\{a, a+1, \dotsc, b-1, b\}$. For an $m \times n$ array $A$, let $F(A)$ be the set of pairs $(i,j)$, with $i \in [1,m]$ and $j \in [1,n]$, such that the $(i,j)$-th cell of $A$ is filled. Then, a \textit{move function} $\phi$ over $A$ is a permutation of $F(A)$, also, 
given $(i,j) \in F(A)$, we say that its \textit{orbit} is the cycle of $\phi$ that contains it. The problem of finding a move function $\phi$ over $A$ that is a permutation of length $|F(A)|$ is denoted by $T(A)$, in this case, we say that $\phi$ is a \textit{solution} to $T(A)$.
If a move function is named (e.g. bishop), we may refer to it omitting ``move function" for brevity.
More in general, we also use the term move function when considering a permutation over a set, especially if the set has a natural association with a set of cells of an array.

Now, given  $(i,j)\in F(A)$, we define the row successor $s_r((i,j))$ to be $(i,j+k)$ where $k\geq 1$ is the minimum integer such that $(i,j+k)\in F(A)$.
Similarly we define the column successor $s_c((i,j))$ to be $(i+k,j)$ where $k\geq1$ is the minimum integer such that $(i+k,j)\in F(A)$.
Given two vectors $R=(r_1,\dots,r_m)\in \{-1,1\}^m$ and $C=(c_1,\dots,c_n)\in \{-1,1\}^n$, we denote $S_{\tiny{R,C}}: F(A)\rightarrow F(A)$ to be the \emph{Crazy Knight's} move function, defined as follows:
$$
S_{R,C}((i,j)):=
\begin{cases}
 s_c(s_r((i,j))) \mbox{ if } r_i=1 \mbox{ and } c_{j'}=1 \mbox{ where } s_r((i,j))= (i,j');\\
 s_c(s_r^{-1}((i,j))) \mbox{ if } r_i=-1 \mbox{ and } c_{j'}=1 \mbox{ where } s_r^{-1}((i,j))= (i,j');\\
 s_c^{-1}(s_r((i,j))) \mbox{ if } r_i=1 \mbox{ and } c_{j'}=-1 \mbox{ where } s_r((i,j))= (i,j');\\
 s_c^{-1}(s_r^{-1}((i,j))) \mbox{ if } r_i=-1 \mbox{ and } c_{j'}=-1 \mbox{ where } s_r^{-1}((i,j))= (i,j').\\
\end{cases}
$$
As explained in \cite{CDP}, one can see the vector $R$ as a choice of the direction of each row: from left to right if $r_i=1$ and from right to left if $r_i=-1$.
Similarly, the vector $C$ can be seen as a choice of the direction of each column: from top to bottom if $c_i=1$ and in the reverse way if $c_i=-1$.
In other words, one can give the following interpretation to  $S_{R,C}((i,j))$: from the position $(i,j)$ we move first in the $i$-th row following the direction of $r_i$ and then, from the arrival position $(i,j')$, we move in the $j'$-th column in direction of $c_{j'}$ arriving in the position $S_{R,C}((i,j))$.
For any $(i,j)\in F(A)$, let $$L(i,j):=((i,j),S_{R,C}((i,j)), S_{R,C}^2((i,j)),\dots, S_{R,C}^{p}((i,j))),$$
where $p$ is the minimum positive integer such that $S_{R,C}^{p+1}((i,j))=(i,j)$,
namely $L(i,j)$ is the orbit of $(i,j)$ with respect to $S_{R,C}$. It is natural to ask whether, in this way, we can cover all the filled positions or not.

\begin{KN}\label{Arco}
Given a partially filled $m\times n$ array $A$, determine if there exist $R\in \{-1,1\}^m$ and $C\in \{-1,1\}^n$ such that for any $(i,j) \in F(A)$,  $L(i,j)$ covers $F(A)$.
\end{KN}
By $P(A)$ we will denote the \probname\ for a given array $A$.
Also, given a filled cell $(i,j)$, if $L(i,j)$ covers all the filled positions of $A$ we will say that the vectors $R$ and $C$ are a solution to $P(A)$.

In Section \ref{sec:prima} we report 
the main known results on the Crazy Knight's Tour Problem  for  cyclically diagonal arrays, namely square arrays whose filled cells are those of consecutive diagonals, and we propose a conjecture. In Section \ref{sec:equivalenza} we show that, under some conditions, the problem is equivalent to a tour problem of another generalized knight on a partially filled toroidal array. This equivalence and some auxiliary results proved in Section \ref{sec:aux_lem} are then used in Section \ref{sec:new_sol} to produce new solutions to $P(A)$ for some infinite classes of cyclically $k$-diagonal square arrays. 

Finally, we conclude with some observations and remarks on the new solutions that may be obtained from the results of the previous sections.

\section{A conjecture for cyclically diagonal square arrays} \label{sec:prima}
The Crazy Knight's Tour Problem has been investigated in several papers \cite{CDY,CDP,CMPP, CP, CPP, DM} in order to get new biembeddings arising from a Heffter array satisfying the additional condition of being \emph{globally simple}, a notion introduced in \cite{CMPP}. In order to study this problem one has to take into account the following necessary condition.
\begin{prop}[\cite{CDP}, Corollary 2.8]\label{prop:all_odd}
Let $A$ be a square array. If there exists a solution to $P(A)$, then $|F(A)|$ is odd.
\end{prop}

Moreover the above necessary condition is also sufficient in the case of completely filled arrays, in fact the following holds.
\begin{thm}[\cite{CDP}, Theorem 3.3]
Let $A$ be a square totally filled array of order $n$. There exists a solution to $P(A)$ if and only if $n$ is odd.
\end{thm}

In the literature, a special attention has been devoted to the research of solution to this problem in the case of cyclically diagonal square matrices because many known Heffter arrays share this structure. Given a square array $A$ of size $n$, for every $i \in [1,n]$ we define the $i$-th diagonal $D_i$ as the set of cells:
\[
D_i = \{(i,1), (i+1,2), \dotsc, (i-1,n)\}.
\] 
Then, an $n\times n$ array $A$ is \textit{cyclically $k$-diagonal} if its filled cells are precisely those of $k$ consecutive diagonals. 
Since the array is toroidal, it is not restrictive to assume that the filled diagonals are the first $k$ ones, namely $F(A) = \bigcup_{i=1}^k D_i$.

In view of some known results and other ones presented in the following sections, we believe that the necessary condition of Proposition \ref{prop:all_odd} is sufficient, more in general, for cyclically $k$-diagonal arrays of size $n$ with $n\geq k\geq3$. 
Note that we require $k\geq 3$ because it is trivial to note that if $A$ is a cyclically $1$-diagonal array, then a solution of $P(A)$ cannot exist.
Hence we propose the following.
\begin{conjecture} \label{conj:crazy_tour}
Let $A$ be a cyclically $k$-diagonal array of order $n$.
 There exists a solution to $P(A)$ if and only if $n$ and $k$ are odd integers with $n\geq k\geq 3$.
\end{conjecture}
The above conjecture holds in the following cases:
\begin{itemize}
\item $k\in[3,200]$, for every odd $n$, see \cite{CDP};
\item $n\geq(k-2)(k-1)$, see \cite{CDP};
\item $\gcd(n,k-1)=1$, see \cite{CMPP}.
\end{itemize}

Furthermore, there is the following  extension theorem, that highly reduces the spectrum of cases to study:
\begin{thm}[\cite{CDP}, Proposition 4.10]  \label{prop:sol_ext}
Let $A$ be a cyclically $k$-diagonal array of size $n > k$. Assume that there exists a solution to $P (A)$ with $R=(1,\dotsc,1)$ or $C = (1,\dotsc, 1)$. Then there exists a solution to $P (A')$ for any cyclically $k$-diagonal array $A'$ of size $n' = n+ \lambda (k-1)$ for any integer $\lambda \geq 0$.
\end{thm}

We remark that the problem $P(A)$ has also been studied for square arrays where the filled cells are in non consecutive diagonals; into details in the following cases:
\begin{itemize}
\item $n \equiv 1 \pmod{4}$, $k=4\ell+3$ where $n\geq k$, and either (a) $n$ prime, (b) $n=k+2$ or (c) $n>(2\ell+2)^2$, and either $n\not\equiv 0 \pmod{3}$ or $\ell\equiv 1 \pmod{3}$, see \cite{CDY};
    \item $n \equiv 1 \pmod{4}$, $k\equiv 3 \pmod{4}$, $n\geq k$ and $3 \leq k \leq 119$, see \cite{CPP}.
\end{itemize}

\section{A restriction of the Crazy Knight's move function} \label{sec:equivalenza}
In this section we show that when the array $A$ has a cyclically $k$-diagonal structure, we can consider a restriction of the Crazy Knight's move function, defined in the Introduction, on a subarray of $A$. The equivalence between the original move function and the restricted one is rather involved to prove, but it allows to get the main results of Section \ref{sec:new_sol}.

For the whole of this section $A$ is a fixed $n \times n$ cyclically $k$-diagonal array, for $n>k\geq 3$ odd integers, and $g=\gcd(n,k-1)$. As already remarked, it is not restrictive to assume that the filled diagonals of $A$ are $D_1,\ldots,D_k$. Here we show that, under some hypotheses on the vectors $R$ and $C$, we can give some necessary and sufficient conditions for the existence of a solution to $P(A)$. In particular, given a set of indexes $E=\{e_1,\dotsc,e_w\}\subset[1,n]$, with $e_1<e_2<\dotsc<e_w$,
we choose $R = (r_1,\dotsc, r_n)\in \{-1,1\}^n$ such that:
\[
r_i = \left\{ \begin{aligned}
&-1 &\text{ if $i\in E$,} \\
&1 &\text{ if $i \not \in E$,} \\
\end{aligned}
\right.
\] 
and $C = (1,\dotsc,1)$.
If $E$ is a set of indexes such that the corresponding sequence $R$ and  $C \equiv 1$ solve $P(A)$, then we say that $E$ \textit{solves} $P(A)$.

For convenience of notation, let $m$ and $\ell$ be the positive integers such that
\begin{equation} \label{eq:param_dim}
	n = mg, \qquad k-1 = \ell g.
\end{equation}
We now construct two permutations, acting on $E$ and $[1,n]$ respectively, in order to produce an equivalent formulation of a particular restriction of the Crazy Knight's move function. We remark that in Example \ref{eq:cyclically_diagonal} we give an explicit construction of what follows, hoping that this might help the reader. The first permutation is defined on the set $E$ as:
\[
\Omega: E \rightarrow E, \quad e_i \mapsto e_{i+(k-1)}.
\]
The second permutation, denoted as $\Theta$, is defined on a partition of $[1,n]$ into intervals of size $g$. For every $j \in [1,m]$ let  $I_j = [1+(j-1)g,jg] $, then:
\begin{equation}  \label{eq:procedureTheta}
\Theta = (I_{1},I_{1-\ell},I_{1-2\ell},\dotsc,I_{1-(m-1)\ell})
\end{equation} 
where the indices are read in their residue class modulo $m$. Note that it can be easily verified from (\ref{eq:param_dim}) that $\Theta$ is a permutation on $\{I_1,\dotsc, I_m\}$.

 With a slight abuse of notation, if $d$ is an integer in $I_{j_t}$ for some $t \in [1,m]$, then by $\Theta(d)$ we mean the element in $I_{j_{t+1}}$ that is equal to $d$ modulo $g$, in other words to the integer $d + n-(k-1)$ modulo $n$.

\begin{prop}
Given a cyclically $k$-diagonal array $A$ of order $n$, let $\pi: D_1 \rightarrow [1,n], (d,d) \mapsto d$ and let $S_{R,C}|_{D_1}: D_1 \rightarrow D_1$ be the restriction on $D_1$ of the Crazy Knight's move function (where, as usual, $D_1$ denotes the main diagonal of $A$). Then, $\pi S_{R,C}|_{D_1}\pi^{-1} \equiv S $,  where: 
\begin{equation} \label{eq:move_fun_diag}
S(d)= \left\{ 
\begin{aligned}
&\Theta(\Omega(d)) \quad\text{ if $d \in E$,} \\
&\Theta(d) \quad\text{ otherwise.} \\
\end{aligned}
\right. 
\end{equation}
\end{prop}
\begin{proof}
If $d \in E$, then  the Crazy Knight's move function maps the cell $(d,d)$ to the  diagonal $D_3$ of $A$. Then to return to the main diagonal $D_1$, we first need to traverse $k-1$ rows such that their relative indexes belong to $E$, which corresponds to the map $\Omega$; we then arrive at the $(k-1)$-th diagonal, and by applying $s_c s_r^{-1}$ we return to the main diagonal. This last operation corresponds to an application of the map $\Theta$, and thus the statement follows.

If $d \not \in E$, then the application of  the Crazy Knight's move function on the cell $(d,d)$ results in the $(d+n-(k-1),d+n-(k-1))$-th cell of $A$, that is the cell having row and column indexes equal to $\Theta(d)$, hence $\pi S_{R,C}|_{D_1}\pi^{-1}(d) =\Theta(d) =S(d)$. 
\end{proof}

\begin{cor}
A set $E \subset [1,n]$ of row indexes solves $P(A)$ if and only if  $S$ is a permutation of $[1,n]$ of maximum length.
\end{cor}

The following remark shows that we may simplify the study of the move function of Equation (\ref{eq:move_fun_diag}) to intervals that are not disjoint from $E$.
\begin{rem} \label{rem:intervals}
Assume that there exists an interval $I_{j_t}$ such that $E \cap I_{j_t} = \emptyset$. Define  the following permutation on the set of intervals $\{I_{j_1},I_{j_2},\dotsc,I_{j_{t-1}},I_{j_{t+1}},\dotsc, I_{j_m}\}$, where $\{j_1,\dotsc, j_m\} = [1,m]$:
\[
\hat{\Theta} = (I_{j_1},I_{j_2},\dotsc,I_{j_{t-1}},I_{j_{t+1}},\dotsc I_{j_m}).
\]
Again, by an abuse of notation if $d$ is an integer in $I_{a}$ for some $a \in [1,m]\setminus\{j_t\}$, then by $\hat{\Theta}(d)$ we mean the element in $\hat{\Theta}(I_{{a}})$ that is equal to $d$ modulo $\gcd(n,k-1)$. Let now $\widehat{S}$ be:
\begin{equation} \label{eq:s_hat}
\widehat{S}(d)= \left\{ 
\begin{aligned}
&\hat{\Theta}(\Omega(d)) \quad\text{ if $d \in E$,} \\
&\hat{\Theta}(d) \quad\text{ otherwise.} \\
\end{aligned}
\right. 
\end{equation}
 Then it easily follows that for every $d \in I_{j_{t-1}}$, $\widehat{S}(d) = S^2(d)$, while for $d \in \{I_{j_1},I_{j_2},\dotsc,I_{j_{t-2}},I_{j_{t+1}},\dotsc, I_{j_m}\}$ the two functions coincide. 
\end{rem}

\begin{ex} \label{eq:cyclically_diagonal}
Take $n=15$, $k=7$ and consider the $7$-cyclically diagonal array of size $15$ shown below, where the gray cells represent the empty cells. Also we take $E = \{1,2,3,7,8\}$
and  show the related horizontal arrows, while every column is pointing downwards, that is $C\equiv 1$. Finally, we have highlighted the  elements $d_1,\dotsc,d_{15}$ of the main diagonal, and we have shown the Crazy Knight's move function applied to $d_{10}$:
\[
\begin{array}{|c||c|c|c|c|c|c|c|c|c|c|c|c|c|c|c|}\hline
\leftarrow&d_1&\cellcolor{gray}&\cellcolor{gray}&\tikzmark{f}{$\cellcolor{gray}$}&\cellcolor{gray}&\cellcolor{gray}&\cellcolor{gray}&\cellcolor{gray}&\cellcolor{gray}&&&&&& \\ \hline
\leftarrow&&d_2&\cellcolor{gray}&\cellcolor{gray}&\cellcolor{gray}&\cellcolor{gray}&\cellcolor{gray}&\cellcolor{gray}&\cellcolor{gray}&\cellcolor{gray}&&&&& \\ \hline
\leftarrow&&&d_3&\cellcolor{gray}&\cellcolor{gray}&\cellcolor{gray}&\cellcolor{gray}&\cellcolor{gray}&\cellcolor{gray}&\cellcolor{gray}&\cellcolor{gray}&&&& \\ \hline
\rightarrow&&&&\tikzmark{g}{$d_4$}&\cellcolor{gray}&\cellcolor{gray}&\cellcolor{gray}&\cellcolor{gray}&\cellcolor{gray}&\cellcolor{gray}&\cellcolor{gray}&\cellcolor{gray}&&& \\ \hline
\rightarrow&&&&&d_5&\cellcolor{gray}&\cellcolor{gray}&\cellcolor{gray}&\cellcolor{gray}&\cellcolor{gray}&\cellcolor{gray}&\cellcolor{gray}&\cellcolor{gray}&& \\ \hline
\rightarrow&&&&&&d_6&\cellcolor{gray}&\cellcolor{gray}&\cellcolor{gray}&\cellcolor{gray}&\cellcolor{gray}&\cellcolor{gray}&\cellcolor{gray}&\cellcolor{gray}& \\ \hline
\leftarrow&&&&&&&d_7&\cellcolor{gray}&\cellcolor{gray}&\cellcolor{gray}&\cellcolor{gray}&\cellcolor{gray}&\cellcolor{gray}&\cellcolor{gray}&\cellcolor{gray} \\ \hline
\leftarrow&\cellcolor{gray}&&&&&&&d_8&\cellcolor{gray}&\cellcolor{gray}&\cellcolor{gray}&\cellcolor{gray}&\cellcolor{gray}&\cellcolor{gray}&\cellcolor{gray}\\ \hline
\rightarrow&\cellcolor{gray}&\cellcolor{gray}&&&&&&&d_9&\cellcolor{gray}&\cellcolor{gray}&\cellcolor{gray}&\cellcolor{gray}&\cellcolor{gray}&\cellcolor{gray}\\ \hline
\rightarrow&\tikzmark{c}{$\cellcolor{gray}$}&\cellcolor{gray}&\cellcolor{gray}&\tikzmark{d}{}&&&&&&\tikzmark{a}{$d_{10}$}&\cellcolor{gray}&\cellcolor{gray}&\cellcolor{gray}&\cellcolor{gray}&\tikzmark{b}{$\cellcolor{gray}$}\\ \hline
\rightarrow&\cellcolor{gray}&\cellcolor{gray}&\cellcolor{gray}&\cellcolor{gray}&&&&&&&d_{11}&\cellcolor{gray}&\cellcolor{gray}&\cellcolor{gray}&\cellcolor{gray}\\ \hline
\rightarrow&\cellcolor{gray}&\cellcolor{gray}&\cellcolor{gray}&\cellcolor{gray}&\cellcolor{gray}&&&&&&&d_{12}&\cellcolor{gray}&\cellcolor{gray}&\cellcolor{gray}\\ \hline
\rightarrow&\cellcolor{gray}&\cellcolor{gray}&\cellcolor{gray}&\cellcolor{gray}&\cellcolor{gray}&\cellcolor{gray}&&&&&&&d_{13}&\cellcolor{gray}&\cellcolor{gray}\\ \hline
\rightarrow&\cellcolor{gray}&\cellcolor{gray}&\cellcolor{gray}&\cellcolor{gray}&\cellcolor{gray}&\cellcolor{gray}&\cellcolor{gray}&&&&&&&d_{14}&\cellcolor{gray}\\ \hline
\rightarrow&\cellcolor{gray}&\cellcolor{gray}&\cellcolor{gray}&\tikzmark{e}{$\cellcolor{gray}$}&\cellcolor{gray}&\cellcolor{gray}&\cellcolor{gray}&\cellcolor{gray}&&&&&&&d_{15}\\ \hline
\end{array}
\link{a}{b}
\link{c}{d}
\link{d}{e}
\link{f}{g}
\]
Since $n=15$ and $k=7$, we have that $g=\gcd(n,k-1)=3$, hence the parameter $\ell$ and $m$ defined in Equation (\ref{eq:param_dim}) are $2$ and $5$, respectively. Following the construction described in  (\ref{eq:procedureTheta})  we partition the interval $[1,15]$ into the sets $I_j = [1+3(j-1),3j]$, for $j \in [1,5]$:
$$I_1 = \{1,2,3\}  \quad I_2 = \{4,5,6\} \quad  I_3 = \{7,8,9\} \quad  I_4 = \{10,11,12\} \quad I_5= \{13,14,15\},$$
and we define the permutation $\Theta$
\[
\Theta = (I_1,I_{1+3},I_{1+6},I_{1+9},I_{1+12})\Rightarrow \Theta = (I_1,I_4,I_2,I_5,I_3).
\]
As described before, if we consider for instance $10 \in I_4$, $10 \not\in E$, then $\Theta(10)$ is the element of $I_2$ that belongs to the same congruence class of $10$ modulo $3$, namely $\Theta(10) = 4$. Indeed, as shown in the array, the Crazy Knight's move function maps $d_{10}$ to $d_4$. 
Note that 
the function $\Omega$ is so defined:
$$\Omega(1)  =2\quad  \Omega(2) =3 \quad \Omega(3) =7\quad \Omega(7)=8\quad \Omega(8) =1.$$
Hence, we can overall define the Crazy Knight's move function restricted to the main diagonal, namely the function $S$,  as in Equation (\ref{eq:move_fun_diag}). In particular, note that for each $d \in I_4$:
\[
d   \xrightarrow{\Theta} \Theta(d) \in I_2 \xrightarrow{\Theta} \Theta^2(d) \in I_5. 
\]
Now, since $E\cap I_2 = \emptyset$, it is clear that $S(d') = \Theta(d')$ for every $d' \in I_2$. Hence, if we let $\hat{\Theta} = (I_1,I_4,I_5,I_3)$, it can be seen that for each $d \in I_4$:
\[
d   \xrightarrow{\hat{\Theta} } \hat{\Theta}(d) = \Theta^2(d).
\]
If we define $\widehat{S}$ as in Equation (\ref{eq:s_hat}), it easily follows that $\widehat{S}(d) = S^2(d)$ for every $d \in I_{4}$, while for $d \in \{I_1, I_3, I_5\}$ the two functions coincide.
\end{ex}

By iteratively applying Remark \ref{rem:intervals} we observe that once we have chosen a set $E$ of indexes such that $r_i=-1$ if and only if $i \in E$, we can focus on the Crazy Knight's move function restricted to the minimal union of the intervals $I_1,\dotsc,I_m$ containing $E$. 

From now on, we assume that the set $E$ is fixed and that for some $t \in [1,m-1]$ the family $\{I_{j_1}, \dotsc, I_{j_{t+1}}\}$ is the minimal set of intervals containing $E$, where the indexes $j_1,\dotsc,j_{t+1}$ are chosen such that $\Theta = (I_{j_1}, \dotsc, I_{j_{t+1}})$.
We moreover require the following conditions: 
\begin{equation}\label{eq:cond1}
    j_1 < j_2 < \dotsc < j_{t+1} 
\end{equation}
\begin{equation} \label{eq:cond2}
\text{$E = \left(\bigcup_{a=1}^{t+1} I_{j_a}\right)  \setminus [j_{t+1}g-f+1, j_{t+1}g]$  for some integer $f \in [1,g-1]$,}
\end{equation}
where, as before, $g=\gcd(n,k-1)$.
In other words, from Condition (\ref{eq:cond1}) we assume that all intervals, whose ordering is such that their indexes have always difference equal to a multiple of $-\ell$ modulo $m$, are also ordered with respect to the classical ordering of the integers.  From Condition (\ref{eq:cond2}), we deduce that $E$ is the union of $t$ intervals $I_{j_1}, \dotsc, I_{j_t}$ of size $g$ and a subinterval of $I_{j_{t+1}}$ of size $g-f$. Let $S_E: E \rightarrow E$ denote the move function restricted on $E$ defined as:
\begin{equation}
\label{eq:s_E}
{S_E}(e)= \Theta^\gamma(\Omega(e)),
\end{equation}
where $\gamma$ is the minimum strictly positive integer such that $\Theta^\gamma(\Omega(e)) \in E$. We now show that under Conditions (\ref{eq:cond1}) and (\ref{eq:cond2}) the function $S_E$ is equivalent to a move function on a partially filled board of another generalization of the classical chessboard knight.

Given a nonzero integer $h$ and two vectors $R,C \equiv 1$, we say that the move function $N_h$ is an \textit{$h$-knight}  on an array $B$ if for every $(i,j)\in F(B)$:
\[
N_h(i,j) := s_C(s_R^h(i,j)).
\]
A $1$-knight is simply called a \textit{bishop}, while a $(-1)$-knight is called a \textit{reversed bishop}. Note that for $h = \pm 1$ the $h$-knight corresponds to the Crazy Knight where $C \equiv 1$ and $R \equiv \pm 1$, respectively.

\begin{ex}
In this example we show a $3$-knight's move over the following array: 
\[
\begin{array}{|c|c|c|c|c|c|} \hline
 \tikzmark{a}{$\bullet$ }&\cellcolor{gray}&\tikzmark{b}{}&\cellcolor{gray}&\tikzmark{c}{ }& \tikzmark{d}{$\diamondsuit$  } \\ \hline
&&&\cellcolor{gray}&\cellcolor{gray}& \cellcolor{gray} \\ \hline
&&&&& \tikzmark{e}{$\circ$ } \\ \hline
\end{array}
\link{a}{b}
\link{b}{c}
\link{d}{e}
\link{c}{d}
\]
In particular, we have that $s_R^3(1,1) = (1,6)$, that is the cell marked with the symbol $\diamondsuit$, and $s_C(1,6) = (3,6)$, namely the cell having the symbol $\circ$. Hence, $N_3(1,1) = (3,6)$.
\end{ex}

\begin{prop} \label{thm:main_eqv}
Let $A$ be a cyclically $k$-diagonal array of size $n$, with $n>k \geq 3$ and let $g= \gcd(n,k-1)$. Let $E$ be a subset of $[1,n]$ satisfying Conditions (\ref{eq:cond1}) and (\ref{eq:cond2}) and set $h=k-1-|E|$. Moreover let $R = (r_1,\dotsc, r_n) \in \{-1,1\}^n$, where $r_i = -1$ if and only if $i \in E$, and let $C = (1,\dotsc,1)$. Then $R$ and $C$ are a solution to $P(A)$ if and only if the $h$-knight is a solution to $T(G)$, where $G:=G_f$ is
\begin{equation} \label{eq:procedureD}
\begin{aligned}
&\text{the\ $(t+1)\times (h+g)$\ array\ whose\ empty\ cells}\\
&\text{are\ exactly\ the\ last\ $f$\ cells\ of\ the\ last\ row.}
\end{aligned}
\end{equation}
\end{prop}
\begin{proof}
We divide the proof of the statement in two steps:  firstly we show  that there is a natural bijection $\phi$ between the set $E$ and the filled cells of a suitable subarray $B$ of $G$; then, we show that for every $e \in E$ it holds:
\begin{equation}\label{eq:equiv_knight_h}
S_E(e) = \phi^{-1}N_h^\alpha\phi(e),
\end{equation}
where $\alpha\geq 1$ is the minimum integer such that $N_h^\alpha(e) \in B$. From Equation (\ref{eq:equiv_knight_h}) it is then clear that the Crazy Knight's move function restricted to the set $E$ is equivalent to an $h$-knight over the array $G$, thus proving the necessary and sufficient condition of the statement.

\paragraph{\textbf{Step 1.}} From Condition (\ref{eq:cond2}) we know that $E$ is the union of $t$ intervals $I_{j_1}, \dotsc, I_{j_t}$ of size $g$, and  a  subinterval $I'_{j_{t+1}} $ of $I_{j_{t+1}}$ of size $g-f$. Let $B $ be the subarray of $G$ containing its last $g$ columns. Hence the empty cells of $B$ are exactly the last $f$ cells of its last row. Note that $|F(B)|=g(t+1)-f=gt+(g-f)=|E|.$

For every $I_{j_a} = \{e_{a,1}, \dotsc, e_{a,g}\}$ with $a \in [1,t]$ and for $I'_{j_{t+1}} = \{e_{t+1,1},\dotsc, e_{t+1,g-f} \}$, where $e_{y,x_1}<e_{y,x_2} \Leftrightarrow x_1<x_2$ for each $y \in [1,t+1]$, let $\phi$ be the map that assigns $e_{u,v}$ to the cell $(u,v)$ of $B$, which is filled from the definition of $B$. Moreover, since the intervals $I_{j_1},\dotsc,I_{j_{t}},I_{j_{t+1}}'$ are pairwise disjoint and $|F(B)| = |E|$, we have that $\phi$ is a bijection. As a final remark, notice that if we endow the indexes of the filled cells of the array $B$ with the following partial ordering:
\[
(u,v) \prec (u',v') \Leftrightarrow u<u',
\]
 the map $\phi$ then preserves the natural ordering of the set $\{j_1, \dotsc, j_{t+1}\}$.

\paragraph{\textbf{Step 2.}} 
By Condition (\ref{eq:cond1}) the sequence $(j_1,\dotsc, j_{t+1})$ is  ordered with respect to both the natural ordering of the integers and  the action of the map $\Omega$. Since $\phi$ is order-preserving, for any $e \in E$ it holds:
\[
	\phi^{-1} s_C \phi (e) = \Theta^\gamma(e).
\]
Thus, recalling the definition of $S_E$ and $N_h$, we note that Equation (\ref{eq:equiv_knight_h}) is equivalent to:
\[
\begin{aligned}
S_E(e) = \phi^{-1}N_h^\alpha\phi(e) &\Leftrightarrow  \Theta^\gamma(\Omega(e))  =  \phi^{-1}s_c s_r^h N_h^{\alpha-1}\phi(e) \\
& \Leftrightarrow \Theta^\gamma(\Omega(e))  =  \phi^{-1}s_c \phi \phi^{-1} s_r^h N_h^{\alpha-1}\phi(e).
\end{aligned}
\]
Hence, in order to prove Equation (\ref{eq:equiv_knight_h}) it is sufficient to verify that:
\begin{equation} \label{eq:simplified}
\phi^{-1} s_r^hN_h^{\alpha-1} \phi (e) = \Omega(e).
\end{equation}

Now, let $G$ be the $(t+1)\times (h+g)$ array defined in (\ref{eq:procedureD}).
Since $h=k-1-|E|$ and $|E|=g(t+1)-f$, we have $h = pg+f$ for some integer $p\geq0$, and  we prove Equation (\ref{eq:simplified}) by induction on $p$.
Assume first that $p=0$, hence $h=f$ and $G$ is a $(t+1)\times (f+g)$ array. Let $e_j \in E$ and $\phi(e_j) = (u,v)$, we distinguish between two cases:
\begin{itemize}
	\item if $u \neq t+1$ and $v\leq g-f$, or if $u = t+1$ and  $v\leq g-2f$, then
	\[
			s_R^f N_f^0\phi(e_j)  = s_R^f\phi(e_j) \in B,
	\]
	hence  $s_R^f\phi(e_j) = \phi(e_{j+f}) = \phi\Omega(e_{j})$;
	\item in the other cases, it is easy to see that:
	\[
s_R^f N_f^1\phi(e_j) \in B,
	\]
	since $s_R^f N_f^1\phi(e_j)$ is the $ (u+2f,v+1)$-th cell in the array $G$, this corresponds to the $(u+f,v+1)$-th cell in the array $B$, that is exactly $\phi\Omega(e_j)$.
\end{itemize}

Assume now that Equation (\ref{eq:simplified}) holds for $h = pg+f$; we prove that this implies that it holds for $h'= (p+1)g+f$. Let $G$ and $G'$ be the two arrays relative to $h$ and $h'$, respectively. It can be easily seen that $G'$ is composed by the concatenation of a completely filled $(t+1)\times g$ array with $G$, hence we refer to $G$ as being contained inside $G'$. Then, let $\Omega$ and $\Omega'$ be the functions on $E$ relative to $h$ and $h'$, respectively.

Assume first that $e \in E$ is such that $s_c \phi\Omega (e)  =\phi \Omega'(e)$.  By inductive hypothesis, on $G$ we have $ s_r^hN_h^{\alpha-1} \phi (e) = \phi\Omega(e)$. 
We point out that in $G$  (resp. $G'$) an $h$-knight (resp. $h'$-knight) move is equivalent to a $(-g)$-knight move for all rows, except for the last one where it corresponds to a $(-g+f)$-knight move.
Let $(u,v) = N_h^{\alpha-1} \phi (e)$ in $G$, that is also $N_{h'}^{\alpha-1} \phi (e)$ in $G'$ by previous remark: it is easy to see that, since $s_r^h(u,v)$ in $G$ belongs to the array $B$, the move $s_r^{h'}(u,v)$ is contained in $G' \setminus G$, hence we apply $N_{h'}$ to $(u,v)$. The resulting cell belongs to $G' \setminus G$, and it follows that $s_r^{h'}N_{h'}(u,v)$ is contained in $B$. Hence, in $G'$, it holds:
\[
s_r^{h'}N_{h'}(u,v) = s_r^{h'}N_{h'}N_{h'}^{\alpha-1}\phi(e) =  s_r^{h'}N_{h'}^{\alpha}\phi(e)  \in B.
\] 
Comparing this formula to the one  for $G$, we have applied  an $h'$-knight move once more. A simple counting argument shows that $s_r^{h'}N_{h'}^{\alpha}\phi(e)$ in $G'$ belongs to the same column of  $ s_r^hN_h^{\alpha-1} \phi (e) = \phi\Omega(e)$ in $G$, and as it is contained in the successive row it follows that:
 \[
 s_r^{h'}N_{h'}^{\alpha}\phi(e) = s_c \phi\Omega (e)  =\phi \Omega'(e).
 \]

Let now $e \in E$ be such that $s_c \phi\Omega (e)  \neq \phi \Omega'(e)$. Notice that this happens if and only if  $\Omega (e) = e_u$, where $u \in[|E|-g+1,|E|]$. Also in this case, let $(u,v) =N_h^{\alpha-1}	\phi(e)$ in $G$, that is also $N_{h'}^{\alpha-1}\phi(e)$ in $G'$. It easily follows that $v \in \{t,t+1\}$, and that $s_r^{h'}(u,v)$ does not belong to $B$, hence we apply the $h'$-knight move function.

If $v = t+1$, then $N_{h'}(u,v) \in G' \setminus G$, in particular, it belongs to the first row of $G'$, hence $s_r^{h'}N_{h'}(u,v) = (x_1,y_1) \in B$. 
If $v = t$, then $N_{h'}(u,v) \in G' \setminus G$; in particular, it belongs to the last row of $G'$. Since $s_c \phi\Omega (e)  \neq \phi \Omega'e)$, it follows that $s_r^{h'}N_{h'}(u,v) \not\in B$, hence we apply the $h'$-knight move function another time, thus $s_r^{h'} N^2_{h'}(u,v) = (x_2,y_2)$ belongs to $B$. Let $\eta $ be the map that associates to each $e \in E $  such that $s_c \phi\Omega (e)  \neq \phi \Omega'(e)$ its corresponding cell, that is either $(x_1,y_1)$ or $(x_2,y_2)$.
It is easy to see that $x_1 = x_2 = 1$, and as $|[|E|-g+1,|E|]| = g$ we have that $\eta$ is a bijection between the elements of $[|E|-g+1,|E|]$ and the first row of $B$. To conclude, it is sufficient to check that $\eta$ is order-preserving, where the ordering of the cells of the row is intended as the natural one, i.e. $(1,y_1) \prec (1,y_2) \Leftrightarrow y_1 <y_2$. 
Let $(u_\gamma,v_\gamma) = N_{h'}^{\alpha_\gamma-1}\phi(e_\gamma)$ for $\gamma \in \{1,2\}$, where $e_1 < e_2$. Clearly, if $u_1 = u_2$, then  the order is preserved. If $u_1 = t$ and $u_2 = t+1$, then $N_h'(u_1,v_1) < (u_2,v_2)$: in fact, $N_h'(u_1,v_1)$ belongs to $G' \setminus G$, while $(u_2,v_2)$ belongs to $G$.
Thus, $\eta$ is order-preserving, concluding the proof.
\end{proof}

\begin{ex} \label{ex:D_example}
Let $n\equiv 5\pmod{10}$ and $k=21$, hence $g=5$. Take $t=2$ and $f = 2$, which implies
$|E|=(t+1)g-f=13$ and $h=7$.
We first show the move function $S_E$ acting on $E$. To this aim, we construct the following array $B$ of order $(t+1)\times g=3\times 5$, that has $|E|=13$ filled cells: 
\[
\begin{array}{|c|c|c|c|c|}\hline
1 &12 &10&13&11  \\ \hline
9&8&7&6&5  \\ \hline
4&3&2&\cellcolor{gray}&\cellcolor{gray}  \\ \hline
\end{array}
\]
A gray cell corresponds to an empty cell, while each filled cell $(i,j)$ is numbered with the smallest integer $\alpha \geq 1$ such that $\phi S_E^{\alpha-1}(e_1)=(i,j)$.
We then construct the array $G$ as in (\ref{eq:procedureD}), and we label the filled cells in the corresponding visiting sequence of the $7$-knight's move function, starting from the cell labelled with $1$:
\[
\begin{array}{|c|c|c|c|c|c|c||c|c|c|c|c|}\hline
22&19&16&13&10&7&4&1&30&25&32&27 \\ \hline
8&5&2&31&26&33&28&23&20&17&14&11 \\ \hline
34&29&24&21&18&15&12&9&6&3&\cellcolor{gray}&\cellcolor{gray} \\ \hline
\end{array}
\]
It can be seen that the restriction of the ordering of the $7$-knight's move function on the last $g$ columns of $G$ (that can be seen by reading the numbers of the last $g$ columns of $G$ in crescent order) corresponds to the ordering  obtained starting from $e_1$ on the array $B$ under the action of $S_E$:
\[
\begin{array}{|c|c|c|c|c|}\hline
1&30&25&32&27 \\ \hline
23&20&17&14&11 \\ \hline
9&6&3&\cellcolor{gray}&\cellcolor{gray} \\ \hline
\end{array}\quad
\Rightarrow \quad\begin{array}{|c|c|c|c|c|}\hline
1 &12 &10&13&11  \\ \hline
9&8&7&6&5  \\ \hline
4&3&2&\cellcolor{gray}&\cellcolor{gray}  \\ \hline
\end{array}
\]
\end{ex}

\section{An equivalent bishop's tour problem} \label{sec:aux_lem}
In this section we prove some auxiliary results that we use in Section \ref{sec:new_sol} to find solutions to the Crazy Knight's Tour Problem on cyclically diagonal square arrays. In particular, our goal is to show that the generalized knight's tour problem on the array $G$ introduced in (\ref{eq:procedureD}) is equivalent to a bishop's tour problem on another array, say $G^*$. To this aim, after having introduced some notation, we present a list of statements that, applied sequentially to $G$, allow to construct the desired array $G^*$. 

 Given two arrays $A$ and $B$, we denote by $A \otimes B$ the Kronecker product between $A$ and $B$. This operation is naturally extended when the two arrays are partially filled: namely, we define the multiplication between an empty cell and another cell (either filled or empty) to be empty.
Moreover, we denote by $J_{m,n}$ the completely filled $m \times n$ array.

We begin with the following simple lemma:
\begin{lem} \label{lem:eq_h_knight_bishop}
Let $G$ be any array. Then, an $h$-knight is a solution to $T(G)$ if and only if a bishop is a solution to $T(G \otimes J_{h,1})$.
\end{lem}

\begin{ex}
In this example we consider a $3$-knight's move on the array $G$, and its corresponding bishop's move on the array $G \otimes J_{3,1}$. It can be seen that a knight's jump on $G$, from the cell labeled $1$ to the one labeled $2$, is equivalent to three consecutive bishop's moves on $G \otimes J_{3,1}$. 
    \begin{figure}[h]\label{fig:correspondence} 
$$
G = \begin{array}{|c|c|c|c|c|} \hline
\tikzmark{a}{1}&\cellcolor{gray}&&& \tikzmark{b}{} \\ \hline 
\SetToWidest{}&\SetToWidest{} &\SetToWidest{}&\SetToWidest{}&\tikzmark{c}{2} \\ \hline
\end{array}
\link{a}{b}
\link{b}{c} \qquad
\Leftrightarrow \qquad
G \otimes J_{3,1} = \begin{array}{|c|c|c|c|c|} \hline
\tikzmark{a1}{1}&\cellcolor{gray}&\tikzmark{b1}{}&&  \\ \hline
&\cellcolor{gray}&\tikzmark{c1}{2}&\tikzmark{d1}{}&  \\ \hline
&\cellcolor{gray}&&\tikzmark{e1}{3}& \tikzmark{f1}{} \\ \hline \hline
&&&&\tikzmark{g1}{4}\\ \hline 
&&&&\\ \hline
\SetToWidest{}&\SetToWidest{}&\SetToWidest{}&\SetToWidest{}&\SetToWidest{}\\ \hline
\end{array}
\link{a1}{b1}
\link{b1}{c1}
\link{c1}{d1} 
\link{d1}{e1} 
\link{e1}{f1} 
\link{f1}{g1} 
$$
\caption{A $3$-knight on $G$ and a bishop on $G\otimes J_{3,1}$.}
\end{figure}
\end{ex}

Given an $\ell\times n$ array $A_1$ and an $m\times n$ array $A_2$, by $\frac{A_1}{A_2}$ we mean the $(\ell+m)\times n$ array whose first
$\ell$ rows are exactly those of $A_1$, while its last $m$  rows are exactly those of $A_2$.
An important tool used in this section is the concept of equivalence of two arrays $A_1$ and $B_1$ with respect to another given array $A_2$. The idea is to establish a condition such that, when considering a move function on $\frac{A_1}{A_2}$, $A_1$ can be replaced with the array $B_1$. 
\begin{defi}
Let $\phi$ be a move function on an array $ \frac{A_1}{A_2}$ and, for each $(i,j) \in F(A_2)$, let $\phi_2: F(A_2) \rightarrow F(A_2)$ be the function:
\begin{equation} \label{eq:phi2}
		\phi_2(i,j) = \phi^k(i,j), 
\end{equation}   
where $k$ is the minimum strictly positive integer such that $\phi^k(i,j) \in F(A_2)$. Similarly, define $\sigma$ to be a move function on an array $\frac{B_1}{A_2}$ and, for each $(i,j) \in F(A_2)$, let $\sigma_2: F(A_2) \rightarrow F(A_2)$ be the function:
\begin{equation} \label{eq:sigma2}
		\sigma_2(i,j) = \sigma^k(i,j), 
\end{equation}
where $k$ is the minimum strictly positive integer such that $\sigma^k(i,j) \in F(A_2)$. We then say that $A_1$ and $B_1$ are \textit{$(\phi,\sigma)$-equivalent with respect to $A_2$} if for every $(i,j) \in F(A_2)$ it holds:
\[
\begin{aligned}
&\phi(i,j)  \in F(A_2)\Leftrightarrow\sigma(i,j)  \in F(A_2), \\
 &\phi(i,j), \sigma(i,j) \not \in F(A_2) \Rightarrow \phi_2(i,j) =\sigma_2(i,j).\\
\end{aligned}
\]
\end{defi}
Moreover, when the move functions $\phi$ and $\sigma$ are essentially the same, e.g. an $h$-knight for some nonzero integer $h$, we simply say that $A_1$ and $B_1$ are \textit{$\phi$-equivalent with respect to $A_2$}, and we  may  omit $A_2$ if it is clear from the context.

\begin{figure}[h]
\begin{center}
\includegraphics[width=0.6\textwidth]{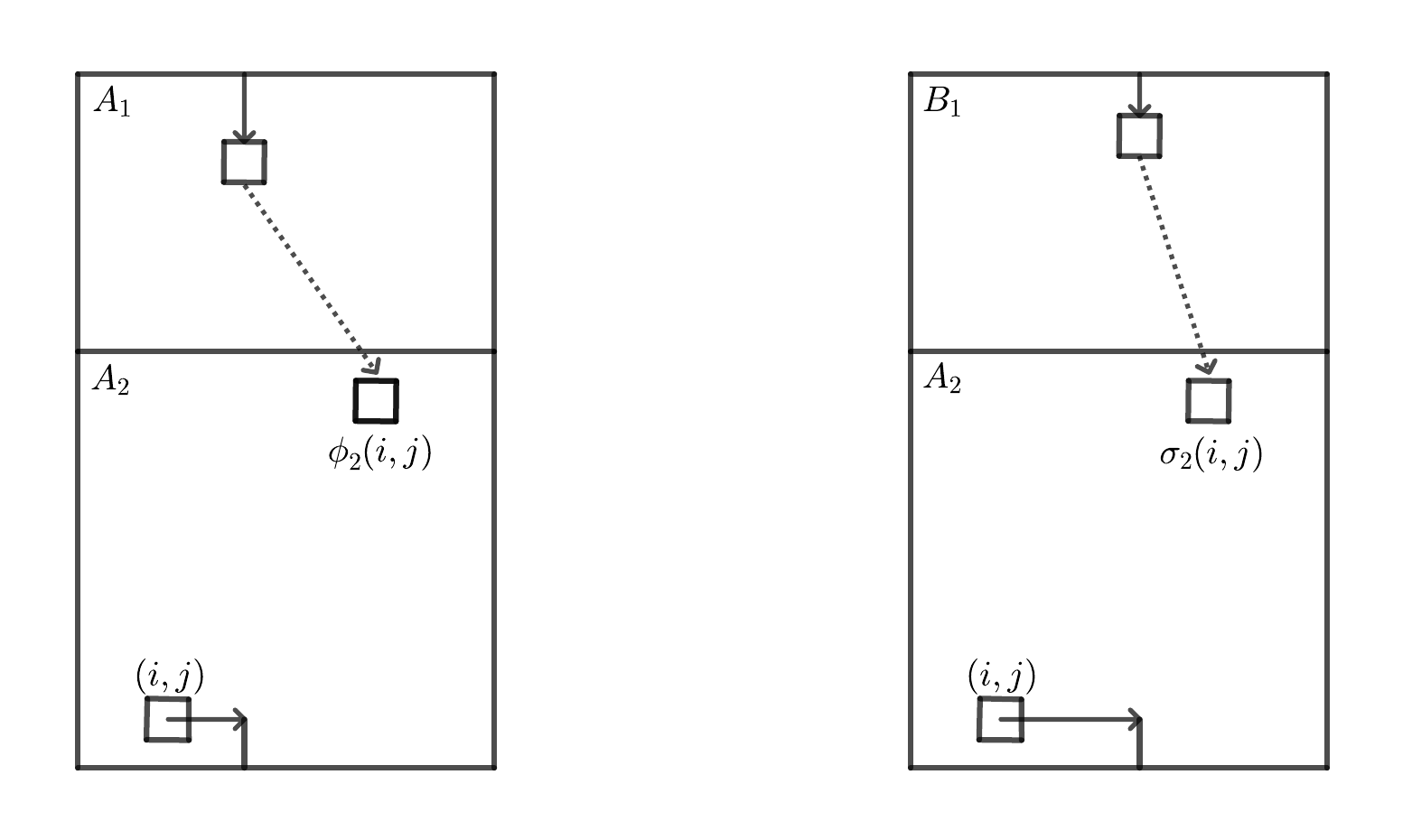}
\caption{The functions $\phi_2$ and $\sigma_2$ applied to the $(i,j)$-th cell of $A_2$.}
\end{center}
\end{figure}
We point out that in order to be equivalent it is not necessary that $A_1$ and $B_1$ have the same order, but only the same number of columns.

\begin{prop}
	Let $\phi$ be the move function of a bishop, and let $A_1$ and $B_1$ be $\phi$-equivalent  with respect to $J_{m,n}$. Then, $A_1$ and $B_1$ are $\phi$-equivalent with respect to every $m \times n$ array. 
\end{prop}
\begin{proof}
Let $A_1$, $B_1$ be as in the statement, set $A_2=J_{m,n}$ and let $B_2$ be any $m \times n$ array. For every $C \in \{A_1,B_1\}$ and $D \in \{A_2,B_2\}$, let:
\begin{itemize}
	\item $\phi_{C,D}$ be the bishop's move function on the array $\frac{C}{D}$;
	\item $\phi_{2,C,D}$ be the function on $D$ defined in Equation (\ref{eq:phi2}) starting from $\phi_{C,D}$.
\end{itemize} 
Assume by contradiction that $A_1$ and $B_1$ are not $\phi$-equivalent with respect to $B_2$, i.e. there exists $(i,j) \in F(B_2)$ such that $\phi_{2,A_1,B_2}(i,j) \neq \phi_{2,B_1,B_2}(i,j)$. Since the bishop's move function is defined as the composition of an horizontal move and a vertical move, we have that:
\[
		\phi_{A_1,B_2} (i,j) = (a_1,a_2) \text{ and } \phi_{B_1,B_2} (i,j) = (b_1,b_2) \Rightarrow a_2 = b_2.
\]
Since the array $A_2$ is completely filled, there exists $(u,v) \in F(A_2)$ such that $\phi_{A_1,A_2}(u,v) = (a_1,a_2)$ and $\phi_{B_1,A_2}(u,v) = (b_1,a_2)$. Now, since $\phi_{2,A_1,A_2} \equiv \phi_{2,B_1,A_2}$, we have that for some $(x,y) \in F(A_2)$ it holds:
\[
	(x,y)=\phi_{2,A_1,A_2}(u,v) = \phi_{2,B_1,A_2}(u,v).
\]
Hence:
\begin{equation} \label{eq:proof_indexes}
\begin{aligned}
		\phi^{-1}_{A_1,A_2}(x,y) &= \phi^{-1}_{A_1,B_2} \phi_{2,A_1,B_2} (i,j), \text{ and} \\
		\phi^{-1}_{B_1,A_2}(x,y) &= \phi^{-1}_{B_1,B_2} \phi_{2,B_1,B_2} (i,j). \\
\end{aligned}		
\end{equation}
If $\phi_{2,A_1,B_2}(i,j) \neq \phi_{2,B_1,B_2}(i,j)$, then in particular $\phi_{2,A_1,B_2}(i,j)$ and $ \phi_{2,B_1,B_2}(i,j)$ have different column indexes (otherwise, by the definition of the bishop's move function, they would be the same cell). However, if we apply an horizontal move $s_{R,A}$ (resp. $s_{R,B}$) to both sides of the first (resp. second) equation in (\ref{eq:proof_indexes}), we can see that the column index must be $y$. We have then proved that $\phi_{2,A_1,B_2}(i,j) = \phi_{2,B_1,B_2}(i,j)$.
\end{proof}

\begin{rem}
If two arrays are $\phi$-equivalent with respect to $J_{1,n}$, where $\phi$ is the bishop's move function, then they are $\phi$-equivalent with respect to every $m \times n$ array, for every $m \geq 1$ (recall that we always assume that every row and every column of an array contains at least one filled cell).
\end{rem}

Let $I_1 \subset [1,m]$ and $I_2 \subset [1,n]$ be intervals. Then, given an $m \times n$ array $A$ such that $F(A) = \{[1,m] \times [1,n]\} \setminus\{ I_1\times I_2\}$, we say that $A$ is an \textit{$(m,n;I_1,I_2)$-holed array} and we denote it  as $A(m,n;I_1,I_2)$.  
Here, we only consider the case where $I_1 = [m-a+1,m]$ and $I_2$  is either $[1,b]$ or $[n-b+1,n]$ for some positive integers $a<m$ and $b<n$. For this choice of $I_1$, we denote $A$ by $A^{L}(m,n;a,b)$ when $I_2=[1,b]$ and by $A^{R}(m,n;a,b)$ when $I_2=[n-b+1,n]$, where $L$ and $R$ stand for \textit{left} and \textit{right}. In both cases the array has exactly $ab$ empty cells.
Note that the array $G$ defined in (\ref{eq:procedureD}) is nothing but an $A^{R}(t+1,h+ g;1,f)$.

\begin{lem}\label{lem:equiv_holed}
Given $a, b$ and $f$ positive integers with $f<\min\{a,b\}$,
let $A_1$ and $B_1$ be an $A^{L}(a,a+b;f,a)$ and an $A^{R}(b+f,a+b;f,b-f)$, respectively. Given a positive integer $d$, set $A_2:=J_{d,a+b}$. Let $\phi$ be the bishop's move function on $\frac{A_1}{A_2}$, and let $\sigma$ be the move function on $\frac{B_1}{A_2}$ so defined: $\sigma$ is the reversed bishop's move function on $B_1$ and the bishop's move function on $A_2$.
 
Then, $A_1$ and $B_1$ are $(\phi,\sigma)$-equivalent with respect to $A_2$.  
\end{lem}
\begin{proof}
Let $A_1, B_1$ and $A_2$ be as in the statement, and let $R_1=(s_1,\dotsc,s_{a+b})$ and $R_d=(u_1,\dotsc,u_{a+b})$ be the first and the last, that is the $d$-th, row of $A_2$, respectively. Since the array $A_2$ is completely filled, the maps $\phi_2$ and $\sigma_2$, defined in Equations (4.1) and (4.2) respectively, are bijections between $R_d$ and $R_1$, in particular:
\[
\phi_2(u_i) = \left\{
\begin{aligned}
	&s_{i+1+a-f}  \quad &\text{ for  $i \in [1, f-1] \cup[b+f, a+b] $,} \\
    	&s_{i+1+a}\quad &\text{  for $i \in [ f, b-1] $, } \\
	&s_{i+1+a-b} \quad&\text{ for $i \in [ b, b+f-1] $.} \\
\end{aligned}
\right.
\]
Similarly, for $\sigma_2(u_i)$:
\[
\sigma_2(u_i) = \left\{
\begin{aligned}
	&s_{i+1-b-f}  \quad &\text{ if $i \in [1, f-1] \cup[b+f, a+b] $,} \\
	&s_{i+1-b} \quad&\text{ for $i \in [ f, b-1] $,} \\
	&s_{i+1+a-b}\quad &\text{for $i \in [ b, b+f-1] $.} \\
\end{aligned}
\right.
\]
Since the indices must be read modulo $a+b$, it follows that $a \equiv -b$, so for each $i \in [1,a+b]$ it holds $\phi_2(u_i) = \sigma_2(u_i)$. 
\end{proof}

\begin{ex} 
Consider the parameters $a = 7$, $b=5$ and $f = 3$.
Here, we respectively show the arrays $A_1$ and $B_1$ of Lemma \ref{lem:equiv_holed}, together with the rows $R_1$ and $R_d$ of $A_2$. In order to better follow the move function we display the last row $R_d$ of $A_2$ as the first row of the arrays  $\frac{A_1}{A_2}$ and  $\frac{B_1}{A_2}$.
If a filled cell contains the number $i$, it means that it is in the orbit of the respective move function of the element $u_{i}$ of $R_d$:
\[
\begin{array}{|c||c|c|c|c|c|c|c|c|c|c|c|c|}\hline
\rightarrow&u_1&u_2&u_3&u_4&u_5&u_6&u_7&u_8&u_9&u_{10}&u_{11}&u_{12} \\ \hline \hline
\rightarrow&12&1&2&3&4&5&6&7&8&9&10&11 \\ \hline 
\rightarrow&11&12&1&2&3&4&5&6&7&8&9&10 \\ \hline 
\rightarrow&10&11&12&1&2&3&4&5&6&7&8&9 \\ \hline 
\rightarrow&9&10&11&12&1&2&3&4&5&6&7&8 \\ \hline 
\rightarrow&\cellcolor{gray}&\cellcolor{gray}&\cellcolor{gray}&\cellcolor{gray}&\cellcolor{gray}&\cellcolor{gray}&\cellcolor{gray}&3&4&5&6&7 \\ \hline 
\rightarrow&\cellcolor{gray}&\cellcolor{gray}&\cellcolor{gray}&\cellcolor{gray}&\cellcolor{gray}&\cellcolor{gray}&\cellcolor{gray}&7&3&4&5&6 \\ \hline 
\rightarrow&\cellcolor{gray}&\cellcolor{gray}&\cellcolor{gray}&\cellcolor{gray}&\cellcolor{gray}&\cellcolor{gray}\cellcolor{gray}&\cellcolor{gray}&6&7&3&4&5  \\ \hline  \hline
&8&9&10&11&12&1&2&5&6&7&3&4 \\ 
&s_1&s_2&s_3&s_4&s_5&s_6&s_7&s_8&s_9&s_{10}&s_{11}&s_{12} \\ \hline 
 & \vdots & \vdots & \vdots & \vdots & \vdots & \vdots & \vdots & \vdots & \vdots & \vdots & \vdots & \vdots   \\ \hline
\end{array}
\]

\[
\begin{array}{|c||c|c|c|c|c|c|c|c|c|c|c|c|}\hline
\rightarrow&u_1&u_2&u_3&u_4&u_5&u_6&u_7&u_8&u_9&u_{10}&u_{11}&u_{12} \\ \hline \hline
\leftarrow&12&1&2&3&4&5&6&7&8&9&10&11 \\ \hline 
\leftarrow&1&2&3&4&5&6&7&8&9&10&11&12 \\ \hline 
\leftarrow&2&3&4&5&6&7&8&9&10&11&12&1 \\ \hline 
\leftarrow&3&4&5&6&7&8&9&10&11&12&1&2 \\ \hline 
\leftarrow&4&5&6&7&8&9&10&11&12&1&2&3 \\ \hline 
\leftarrow&5&6&7&8&9&10&11&12&1&2&\cellcolor{gray}&\cellcolor{gray} \\ \hline 
\leftarrow&6&7&8&9&10&11&12&1&2&5&\cellcolor{gray}&\cellcolor{gray} \\ \hline
\leftarrow&7&8&9&10&11&12&1&2&5&6&\cellcolor{gray}&\cellcolor{gray} \\ \hline\hline
&8&9&10&11&12&1&2&5&6&7&3&4 \\ 
&s_1&s_2&s_3&s_4&s_5&s_6&s_7&s_8&s_9&s_{10}&s_{11}&s_{12} \\ \hline  
 & \vdots & \vdots & \vdots & \vdots & \vdots & \vdots & \vdots & \vdots & \vdots & \vdots & \vdots & \vdots   \\ \hline
\end{array}
\]
The $(\phi,\sigma)$-equivalence between $A_1$ and $B_1$ immediately follows.
\end{ex}

\begin{cor}\label{cor:equiv_compl_filled}
For every triple of strictly positive integers $(a,b,d)$ set  $A_1:= J_{a,a+b}$,  $B_1:=J_{b,a+b}$ and $A_2:= J_{d,a+b}$. Let $\phi$ be the bishop's move function on 
$\frac{A_1}{A_2}$, and let $\sigma$ be the move function on $\frac{B_1}{A_2}$ so defined: it is the reversed bishop's move function on $B_1$ and
the bishop's move function on $A_2$.

Then, $A_1$ and $B_1$ are $(\phi,\sigma)$-equivalent with respect to $A_2$.  
\end{cor}

\begin{ex}
Let $a = 7$ and $b = 5$, we show that $J_{7,12}$ and  $J_{5,12}$ are $(\phi,\sigma)$-equivalent with respect to $J_{d,12}$ for every $d\geq1$:
\[
\begin{array}{|c||c|c|c|c|c|c|c|c|c|c|c|c|}\hline
\rightarrow&u_1&u_2&u_3&u_4&u_5&u_6&u_7&u_8&u_9&u_{10}&u_{11}&u_{12}  \\ \hline \hline
\rightarrow&12&1&2&3&4&5&6&7&8&9&10&11 \\ \hline
\rightarrow&11&12&1&2&3&4&5&6&7&8&9&10\\ \hline
\rightarrow&10&11&12&1&2&3&4&5&6&7&8&9\\ \hline
\rightarrow&9&10&11&12&1&2&3&4&5&6&7&8\\ \hline
\rightarrow&8&9&10&11&12&1&2&3&4&5&6&7\\ \hline 
\rightarrow&7&8&9&10&11&12&1&2&3&4&5&6\\ \hline
 \rightarrow&6&7&8&9&10&11&12&1&2&3&4&5\\ \hline \hline
&5&6&7&8&9&10&11&12&1&2&3&4\\
&s_1&s_2&s_3&s_4&s_5&s_6&s_7&s_8&s_9&s_{10}&s_{11}&s_{12} \\ \hline 
 & \vdots & \vdots & \vdots & \vdots & \vdots & \vdots & \vdots & \vdots & \vdots & \vdots & \vdots & \vdots   \\ \hline
\end{array}
\]

\[
\begin{array}{|c||c|c|c|c|c|c|c|c|c|c|c|c|}\hline
\rightarrow&u_1&u_2&u_3&u_4&u_5&u_6&u_7&u_8&u_9&u_{10}&u_{11}&u_{12}  \\ \hline \hline
\leftarrow &12&1&2&3&4&5&6&7&8&9&10&11 \\ \hline
\leftarrow&1&2&3&4&5&6&7&8&9&10&11&12 \\ \hline 
\leftarrow&2&3&4&5&6&7&8&9&10&11&12&1 \\ \hline
\leftarrow&3&4&5&6&7&8&9&10&11&12&1&2 \\ \hline
\leftarrow&4&5&6&7&8&9&10&11&12&1&2&3 \\ \hline \hline
             &5&6&7&8&9&10&11&12&1&2&3&4 \\ 
&s_1&s_2&s_3&s_4&s_5&s_6&s_7&s_8&s_9&s_{10}&s_{11}&s_{12} \\\hline
 & \vdots & \vdots & \vdots & \vdots & \vdots & \vdots & \vdots & \vdots & \vdots & \vdots & \vdots & \vdots   \\ \hline
\end{array}
\]
\end{ex}

\begin{lem} \label{lemma:equiv1}
Let $G$ be an $A^{R}((t+1)h, h+g; h,f)$, where $t \geq 1$, $f<g$ and $f \leq h$. Then a bishop is a solution to $T(G)$ if and only if it is a solution to $T(G')$, where $G'$ is an $A^{L}(th,h+g; g-f,h)$.
\end{lem} 
\begin{proof}
Consider the subarray $G_1$ of $G$ containing its last $h+g-f+1$ rows, that is an $A^{R}(h+g-f+1,h+g; h,f)$, and the subarray $G_1'$ of $G'$ containing its last $g-f+1$ rows, that is a $A^{L}(g-f+1,h+g;g-f,h)$. Let $\phi$ be the bishop's move function, we show that $G_1$ and $G_1'$ are $\phi$-equivalent with respect to
$J := J_{th-g+f-1,h+g}$; note that 
$G=\frac{J}{G_1}$ and $G'=\frac{J}{G_1'}$.

Let $R_1 =(s_1, \dotsc, s_{h+g}) $ and $R_{(t+1)h} = (u_1,\dotsc, u_{h+g})$ be the first and the last row of $G$, respectively. We then have:
\[
\phi_2(u_i) = \left\{
\begin{aligned}
&s_{i+2}  \quad &\text{for $i\in[1,h-2] \cup [h+g-1,h+g]$,} \\ 
&s_{i+g-f+2}   \quad &\text{for $i\in [h-1, h+f-2]$,} \\
&s_{i+h+g-f+2} \quad &\text{for $i\in [h+f-1, h+g-2]$.} \\
\end{aligned}
\right.
\]

Let then $R_1' = (s'_1, \dotsc, s'_{h+g}) $ and $R_{th}' =(u_1',\dotsc, u_{h+g}') $  be the first and the last row of $G'$, respectively. We have:
\[
\phi_2(u_i') = \left\{
\begin{aligned}
&s_{i+2}'  \quad &\text{for  $i\in[1,h-2] \cup [h+g-1,h+g]$,} \\
&s_{i+g-f+2}' \quad &\text{ for $i\in [h-1, h+f-2]$, } \\
&s_{i-f+2}' \quad &\text{for $i\in [h+f-1, h+g-2]$.} \\
\end{aligned}
\right.
\]
Recalling that the arrays have $h+g$ columns, and indexes have to be read modulo $h+g$, we get that:
\[
\phi_2(u_i) = s_j \quad \Leftrightarrow \quad \phi_2(u_i')=s_j', 
\]
hence, $G_1$ and $G_1'$ are $\phi$-equivalent with respect to any completely filled array, thus the statement follows.
\end{proof}

\begin{ex}
Let  $t = 1$, $h = 7$, $g = 5$ and $f= 3$, then $G$ and $G'$ are an $A^{R}(14, 12; 7,3)$ and an $A^{L}(7,12; 2,7)$, respectively. Here, we show the subarrays $G_1$ and $G_1'$, that are respectively an $A^{R}(10,12;7,3)$ and an $A^{L}(3,12;2,7)$, together with $R_1,R_{(t+1)h}$ and $R_1',R_{th}'$:
\[
\begin{array}{|c|c|c|c|c|c|c|c|c|c|c|c|c|}\hline
&u_1&u_2&u_3&u_4&u_5&u_6&u_7&u_8&u_9&u_{10}&u_{11}&u_{12}  \\ \hline \hline
\rightarrow&12&1&2&3&4&5&6&7&8&9&10&11 \\ \hline 
\rightarrow&11&12&1&2&3&4&5&6&7&8&9&10 \\ \hline 
\rightarrow&10&11&12&1&2&3&4&5&6&7&8&9 \\ \hline 
\rightarrow&9&10&11&12&1&2&3&4&5&\cellcolor{gray}&\cellcolor{gray}&\cellcolor{gray} \\ \hline 
\rightarrow&5&9&10&11&12&1&2&3&4&\cellcolor{gray}&\cellcolor{gray}&\cellcolor{gray} \\ \hline 
\rightarrow&4&5&9&10&11&12&1&2&3&\cellcolor{gray}&\cellcolor{gray}&\cellcolor{gray} \\ \hline 
\rightarrow&3&4&5&9&10&11&12&1&2&\cellcolor{gray}&\cellcolor{gray}&\cellcolor{gray} \\ \hline 
\rightarrow&2&3&4&5&9&10&11&12&1&\cellcolor{gray}&\cellcolor{gray}&\cellcolor{gray} \\ \hline 
\rightarrow&1&2&3&4&5&9&10&11&12&\cellcolor{gray}&\cellcolor{gray}&\cellcolor{gray} \\ \hline 
\rightarrow&12&1&2&3&4&5&9&10&11&\cellcolor{gray}&\cellcolor{gray}&\cellcolor{gray} \\ \hline  \hline
&11&12&1&2&3&4&5&9&10&6&7&8\\ 
&s_1&s_2&s_3&s_4&s_5&s_6&s_7&s_8&s_9&s_{10}&s_{11}&s_{12} \\\hline 
  &\vdots & \vdots & \vdots & \vdots & \vdots & \vdots & \vdots & \vdots & \vdots & \vdots & \vdots & \vdots   \\ \hline
\end{array}
\]
\[
\begin{array}{|c|c|c|c|c|c|c|c|c|c|c|c|c|}\hline
&u_1'&u_2'&u_3'&u_4'&u_5'&u_6'&u_7'&u_8'&u_9'&u_{10}'&u_{11}'&u_{12}'   \\ \hline \hline
\rightarrow&12&1&2&3&4&5&6&7&8&9&10&11 \\ \hline 
\rightarrow&\cellcolor{gray}&\cellcolor{gray}&\cellcolor{gray}&\cellcolor{gray}&\cellcolor{gray}&\cellcolor{gray}& \cellcolor{gray} &6&7&8&9&10 \\ \hline
\rightarrow&\cellcolor{gray}&\cellcolor{gray}&\cellcolor{gray}&\cellcolor{gray}&\cellcolor{gray}&\cellcolor{gray}& \cellcolor{gray}&10 &6&7&8&9 \\ \hline
\hline
&11&12&1&2&3&4&5&9&10&6&7&8\\ &s_1'&s_2'&s_3'&s_4'&s_5'&s_6'&s_7'&s_8'&s_9'&s_{10}'&s_{11}'&s_{12}' \\\hline 
  &\vdots & \vdots & \vdots & \vdots & \vdots & \vdots & \vdots & \vdots & \vdots & \vdots & \vdots & \vdots   \\ \hline
\end{array}
\]
It can then be seen that $G_1$ and $G_1'$ are $\phi$-equivalent, where $\phi$ is the bishop's move function.
\end{ex}

\begin{lem}\label{lem:equivalenza2}
Let $G'$ and $G^*$ be an $A^L(th,h+g; g-f,h)$ and an $A^R((t+1)g-f,h+g; g-f,f)$, respectively, where $t\geq1$ and $h>g-f$. Then a bishop is a solution to $T(G')$ if and only if it is a solution to $T(G^*)$.
\end{lem}
\begin{proof}
Let $G_1'$ be the subarray of $G'$ containing its last $h$ rows, that is an $A^L(h,h+g;g-f,h)$; by Lemma \ref{lem:equiv_holed} a bishop in $G_1'$ is equivalent to a reversed bishop in an $A^R(2g-f,h+g;g-f,f)$ that we denote as $G_1^*$.

The remaining rows of the array $G'$ form a $J_{(t-1)h,h+g}$, say $G_2'$, and by Corollary \ref{cor:equiv_compl_filled} a bishop in each set of $h$ rows of $G_2'$ is equivalent to a reversed bishop in a $J_{g,h+g}$. Hence, overall, a bishop in $G_2'$ is equivalent to a reversed bishop in a $J_{(t-1)g,g+h}$, which we denote by $G_2^*$.

The array $\frac{G_2^*}{G_1^*}$ is an $A^R((t+1)g-f,h+g; g-f,f)$, which we denote by $G^*$, and the move function is a reversed bishop on $G^*$. Since clearly a reversed bishop is a solution to $T(G^*)$ if and only if a standard bishop solves $T(G^*)$, we have proven the statement.
\end{proof}

\begin{prop} \label{thm:eqv_array}
Let $G$ and $G^*$ be an $A^{R}(t+1,h+g;1,f)$ and an $A^R((t+1)g-f,h+g;g-f,f)$, respectively.
A $h$-knight is a solution to $T(G)$ if and only if a bishop is a solution of $T(G^*)$.
\end{prop}
\begin{proof}
     By Lemma \ref{lem:eq_h_knight_bishop}, a $h$-knight is a solution to $T(G)$ if and only if a bishop is a solution to $T(A^R)$, where $A^R=A^R((t+1)h,h+g;h,f)$ and, by Lemma \ref{lemma:equiv1}, this happens if and only if a bishop is a solution to $T(A^L)$ where $A^L=A^L(th,h+g;g-f,h)$. Finally, by Lemma \ref{lem:equivalenza2} this last condition is equivalent to a bishop's tour over $A^R((t+1)g-f,h+g;g-f,f)$.
\end{proof}

We conclude this section with another auxiliary lemma, which will be useful to get the main results:
\begin{lem} \label{lem:bishop_tour}
Let $A$ be an $A^R(m,n;a,b)$. If $\gcd(m,n) = a+b$ and $\gcd(a,b) = 1$, then a bishop is a solution to $T(A)$. 
\end{lem}
\begin{proof}
It has been proven in \cite{DF} that to cover each filled cell of $J_{m,n}$ once we need exactly the orbits of $\gcd(m,n) = a+b$ bishops.
Let then $\mathcal{B}= \{B_1, B_2,\dotsc,B_{a+b}\}$ be such a set of bishops. Note that if we fix $a+b$ consecutive cells in a row of $J_{m,n}$, then each cell belongs to the orbit of a different bishop. 

Let $(z_1, z_2,\dotsc, z_{a+b})$ be consecutive cells in the $(m-a)$-th row of $A$, so that $z_{a+b} = (m-a,n-1)$, and place the indexes of the bishops so that the cell $z_i$ belongs to the orbit of $B_i$. Let $\phi$ denote the bishop's move function, and define $\eta: [1,a+b] \rightarrow [1,a+b]$ to be the following map:
\[
 \eta(i) = j \Leftrightarrow \text{  $\phi^q(z_i)$ is in the orbit of $B_j$,}
\]
where $q$ is the minimum strictly positive integer such that $\phi^q(z_i)$ is either in the first column or in the first row of $A$.
It is not hard to verify that:
\[
\eta(i) = \left\{
\begin{aligned}
	i-a	& \text{ if $i-a\geq 1$,} \\
	i+b	& \text{ otherwise.}
\end{aligned}
\right.
\]
We now prove that $\eta$ is a permutation on $[1,a+b]$ of maximum length. We have that $\eta(1)=b+1$, and through successive applications of $\eta$ we cover all the elements in $[a,a+b]$ that are equal to $b+1$ modulo $a$; in particular, the last visited element that belongs to this congruence class is precisely  the residue class of $b+1 $ modulo $a$. Proceeding inductively, it follows that we cover all the elements that are equal to $\ell b+1$ modulo $a$, for $\ell \in[1,a]$; since $\gcd(a,b)=1$, this corresponds to the whole interval $[1,a+b]$. Thus, $\eta$ is a permutation of $[1,a+b]$ of maximum length.

This proves that a bishop is a solution to $T(A)$: indeed, if we start from the cell $z_i$, after $\phi^q$ steps (where $q$ is defined as before) we arrive in the orbit of the bishop $B_{\eta(i)}$. Then, we visit all the cells that belong to the orbit of $B_{\eta(i)}$, and we arrive at the cell labelled as $z_{\eta(i)}$. By repeating the iteration inductively, it immediately follows that the whole array $A$ is covered by the orbit of one bishop, hence the statement.
\end{proof}

\begin{ex}
Let $a = 3$ and $b=13$, and choose $m$ and $n$ such that $\gcd(m,n) = 16$; construct then the array $A$, that is an $A^{R}(m,n;3,13)$. Let $\mathcal{B} = \{B_1, \dotsc, B_{16}\}$ be a set of $16$ bishops such that they cover the completely filled array $J_{m,n}$. Let then $(z_1,\dotsc,z_{16})$ be consecutive cells in the $(m-a)$-th row of $A$ such that $z_{16} = (m-a,n-1)$. Here, we represent the cells of $A$ having row index in $\{1\} \cup [m-3,m]$ and column index in $\{1\} \cup [n-16,n]$, where the first row and the first column are display as the last one. Moreover, if a filled cell contains the number $i$, it means that it is contained in the orbit of the bishop $B_i$ in the array $J_{m,n}$ (while every filled cell belongs to the orbit of some bishop $B_i$, we only write the numbers in the cells that are relevant in the discussion):
\[
\small{
\begin{array}{|c|c|c|c|c|c|c|c|c|c|c|c|c|c|c|c|c|c|} \hline
\tikzmark{p}{$z_1$}&z_2&z_3&z_4&z_5&z_6&z_7&z_8&z_9&z_{10}&z_{11}&z_{12}&z_{13}&\tikzmark{d}{$z_{14}$}&\tikzmark{e}{$z_{15}$}&z_{16} && \\ \hline
&1&2&3&\cellcolor{gray}&\cellcolor{gray}&\cellcolor{gray}&\cellcolor{gray}&\cellcolor{gray}&\cellcolor{gray}&\cellcolor{gray}&\cellcolor{gray}&\cellcolor{gray}&\cellcolor{gray}&\cellcolor{gray}&\cellcolor{gray}&\cellcolor{gray}& \\ \hline
&&1&2&\cellcolor{gray}&\cellcolor{gray}&\cellcolor{gray}&\cellcolor{gray}&\cellcolor{gray}&\cellcolor{gray}&\cellcolor{gray}&\cellcolor{gray}&\cellcolor{gray}&\cellcolor{gray}&\cellcolor{gray}&\cellcolor{gray}&\cellcolor{gray}&16 \\ \hline
&&&\tikzmark{a}{1}&\cellcolor{gray}&\cellcolor{gray}&\cellcolor{gray}&\cellcolor{gray}&\cellcolor{gray}&\cellcolor{gray}&\cellcolor{gray}&\cellcolor{gray}&\cellcolor{gray}&\cellcolor{gray}&\cellcolor{gray}&\cellcolor{gray}&\cellcolor{gray}& \tikzmark{b}{15} \\ \hline
&&&&1&2&3&4&5&6&7&8&9&10&\tikzmark{f}{11} &12&13&\tikzmark{c}{14} \\ \hline
\end{array}
\link{a}{b}
\link{b}{c}
\link{d}{e}
\link{e}{f}
}
\]
We can then write the function $\eta$, that here we directly show as a permutation of $[1,16]$:
\[
\eta = (1,14,11,8,5,2,15,12,9,6,3,16,13,10,7,4).
\]
Hence, for instance, if we start from $z_1$, after applying $q=4$ times the bishop's move function, we land in a cell that belongs to the orbit of the bishop $B_{\eta(1)} = B_{14}$. From there, we move through every cell contained in the orbit of $B_{14}$, and we arrive in $z_{14}$. We then repeat the process with $B_{\eta(14)}=B_{11}$, and so on, thus visiting every filled cell of the array $A$.
\end{ex}

\section{Main results for cyclically diagonal arrays} \label{sec:new_sol}
In this section we use Propositions \ref{thm:main_eqv} and \ref{thm:eqv_array} to construct new solutions to the Crazy Knight's Tour Problem for cyclically diagonal square arrays. 
Since a solution to $P(A)$ for a cyclically $k$-diagonal array $A$ of order $n$ with $g = \gcd(n,k-1)=1$ has been established in  \cite{CDP}, in this section we implicitly assume that $g\ \geq 3$.

\begin{thm} \label{prop:case1}
Let  $A$ be a cyclically $k$-diagonal array of size $n$ and set $g = \gcd(n,k-1)$.
If $\gcd(g+1,k-1)=2$, then there exists a solution to $P(A)$.
\end{thm}
\begin{proof}
Choose $E\subset [1,n]$ such that $|E| = g+1$ and Conditions (\ref{eq:cond1}) and (\ref{eq:cond2}) are satisfied (note that this can always be done). 
Following  the notation of Proposition \ref{thm:main_eqv}, we have $h = k-2-g$ and $f =g-1$.  We may assume that $k\geq 9$, since for $k \in\{3,5,7\}$ the solution is already known from \cite{CDP}, hence $h \geq 3$. From Proposition \ref{thm:main_eqv} and Lemma \ref{lemma:equiv1}, we know that $E$ solves the Crazy Knight's Tour Problem if and only if a bishop is a solution to $T(G^*)$, where $G^*$ is an $A^{L}(h,h+g;1,h)$. 

Now, let $B$ be the subarray of $G^*$ containing its first $h$ columns, and  define  $Z = (z_1, \dotsc, z_h)$ and $Z' = (z_1', \dotsc, z_h')$ to be the $(h+1)$-th and the $(h+g)$-th columns of $G^*$, respectively. Finally, let $\phi$ be the bishop's move function on $G$, and $\phi_2$ be the map from $Z'$ to $Z$ such that $\phi_2(z_i') = z_j$,
where $z_j$ is the first cell of $Z$ that is visited by the move function after $z_i'$. 

It is easy to see that for each $i \in [1,h-2]$ we have $\phi_2(z_i') = z_{i+2}$, while $\phi_2(z_h') = z_1$ and $\phi_2(z_{h-1}') = z_2$. Hence, if $G^*_1$ is the array obtained from $G^*$ by replacing $B$ with an $A^{L}(h,1;1,1)$, a bishop is a solution of $T(G^*)$ if and only if it solves $T(G^*_1)$. We have that $G^*_1$ is an $A(g+1,k-2-g;1,1)$, hence the statement  follows from Lemma \ref{lem:bishop_tour}.
\end{proof}

\begin{ex}
Take  $n\equiv7\pmod{14}$, $n \geq 21$ and $k=15$. Then:
\[
g = \gcd(n,k-1) = 7 \qquad h=k-2-g =6  \qquad f =g-1=6.
\]
Note that $\gcd(g+1,k-1)=2$. We construct the array $G^*$, that is an $A^{L}(6,13;1,6)$. Moreover, we show the orbits of $z_2'$ and $z_5'$ under the action of the move function $\phi$, where the visited cells are marked with the corresponding index of $z_j'$:
\[
\begin{array}{|c|c|c|c|c|c||c|c|c|c|c|c|c|} \hline
5&&&2&&5&z_1&&&&&&z_1' \\ \hline
&5&&&2&&z_2&&&&&&z_2' \\ \hline
2&&5&&&2&z_3&&&&&&z_3' \\ \hline
&2&&5&&&z_4&&&&&&z_4' \\ \hline
&&2&&5&&z_5&&&&&&z_5' \\ \hline
\cellcolor{gray}&\cellcolor{gray}&\cellcolor{gray}&\cellcolor{gray}&\cellcolor{gray}&\cellcolor{gray}&z_6&&&&&&z_6' \\ \hline
\end{array}
\]
Hence, we have $\phi(z_2')=z_4$ and $\phi(z_5')=z_2$. It can be seen that we can equivalently consider the following array $G_1^*$, where we have replaced the first $h$ columns of $G^*$ with an $A^{L}(6,1;1,1)$:
\[
\begin{array}{|c||c|c|c|c|c|c|c|} \hline
5&z_1&&&&&&z_1' \\ \hline
&z_2&&&&&&z_2' \\ \hline
2&z_3&&&&&&z_3' \\ \hline
&z_4&&&&&&z_4' \\ \hline
&z_5&&&&&&z_5' \\ \hline
\cellcolor{gray}&z_6&&&&&&z_6' \\ \hline
\end{array}
\]
\end{ex}

\begin{cor} \label{cor:case1}
Let $n=mg$ be an odd integer and let $\ell\in [2,m-1]$ be an even integer such that
 $\gcd(g+1,\ell )=2$ and $\gcd(m,\ell)=1$.
Then there exists a solution to $P(A)$, where $A$ is a cyclically $(1+ \ell g)$-diagonal $n \times n$ array.
\end{cor}
\begin{proof}
Let  $k=1+ \ell g$, since $\gcd(m,\ell)=1$, $g$ is nothing but $\gcd(n,k-1)$.
We will prove that $\gcd(g+1,k-1)=2$, then the result will follow from Theorem \ref{prop:case1}.
Clearly $\gcd(g+1,k-1)=\gcd(g+1,k-2-g)$. Since $g$ is odd,  $g+1=2q$ for some positive integer $q$, then:
\[
k-2-g=\ell g - (g+1) = \ell (2q-1) - 2q = 2 \ell q - \ell - 2q.
\]
Hence $\gcd(g+1,k-1)=\gcd(2q, 2\ell q - \ell - 2q) = \gcd(2q,\ell)=\gcd(g+1,\ell)$ and this is $2$ by hypothesis.
\end{proof}

\begin{thm} \label{prop:case2}
Let $A$ be a cyclically $k$-diagonal array of order $n$, with $n$ and $k$ odd, and set $g=\gcd(n,k-1)$.
If there exists $E \subset[1,n]$ satisfying Conditions (\ref{eq:cond1}) and (\ref{eq:cond2}) with $|E|= \frac{k-1}{2} + \frac{g-1}{2}$, then there exists a solution to $P(A)$.
\end{thm}
\begin{proof}
Let $E$ be as in the statement. By Propositions \ref{thm:main_eqv} and \ref{thm:eqv_array} we have that a solution to $P(A)$ exists if and only if a bishop is a solution to $T(G^*)$, where $G^*$ is an $A^{R}(\frac{k-1}{2} + \frac{g-1}{2}, \frac{k-1}{2} + \frac{g+1}{2}; \frac{g-1}{2}, \frac{g+1}{2})$. Let $B$ be the subarray of $G^*$ containing its first $\frac{k-1}{2}$ columns. Let $Z =(z_1, \dotsc, z_{\frac{k-1}{2}})$ and $Z' =(z_1', \dotsc, z_{\frac{k-1}{2}}')$ be the ordered sequence of the cells of the $(\frac{k+1}{2})$-th column and the $(\frac{k-1}{2} + \frac{g+1}{2})$-th column  of $G^*$, respectively. 

 Let $\phi$ be the bishop's move function on $G^*$, and let $\phi_2$ be the map from $Z'$ to $Z$ such that
$\phi_2(z_i') = z_j$, where $z_j$ is the first cell of $Z$ that is visited by the move function after $z_i'$. 

We then distinguish between three different cases:
\begin{itemize}
	\item if $i \in[1,\frac{g-3}{2}]$, then we need to apply $2 \frac{k-1}{2}+1 = k$ times $\phi$, hence the element $z_i'$ is mapped by $\phi_2$ to the element $z_j$, where $j = i+2-g$;
	\item if $i = \frac{g-1}{2}$, then we apply $k$ times the function $\phi$; however, we have to consider that we need to remove $\frac{k-1}{2}$, hence the element $z_i'$ is mapped by $\phi_2$ to $z_1$;
    	\item if $i \in[\frac{g+1}{2}, \frac{k-1}{2}]$, in order to arrive to $Z$ we need to apply $\frac{k+1}{2}$ times the function $\phi$. Since the column $Z$ has empty cells, the element $z_i'$ is mapped by $\phi_2$ to the element $z_j$, where $j = i+ \frac{k+1}{2}-\frac{g-1}{2}$. 
\end{itemize}
Now, if we apply from $Z$ the bishop's move function $\frac{g+1}{2}$ times, we arrive again at the column $Z'$; it is not hard to see that
$ \phi^{\frac{g+1}{2}} 	\phi_2(z_i') = z_{i+1}'$  for every $i \in [1,\frac{k-1}{2}]$.

Note that this is exactly equivalent to a bishop's move function on  $J_{\frac{k-1}{2},1}$, that is the column $Z$. This proves that a bishop is a solution to $T(G^*)$ if and only if it is a solution to $T(Z)$, and it is trivial to see that this is always the case. 
\end{proof}

\begin{ex}
Let $n\equiv 7\pmod {14}$, $n \geq 35$, and $k = 29$, then $g = 7$ and $G^*$ is an $A^{R}(17,18;3,4)$. Here, we show the array $G^*$ and we write the elements in the $15$-th and the $18$-th column. Moreover, we show the orbits of $z_1'$, $z_3'$ and $z_9'$ under the action of the move function $\phi$, where the visited cells are marked with the corresponding index of $z_j'$:
\[
\begin{array}{|c|c|c|c|c|c|c|c|c|c|c|c|c|c||c|c|c|c|} \hline
3&&1&&&&&&9&&&&&&z_1&&&z_1' \\ \hline
1&3&&1&&&&&&9&&&&&z_2&&&z_2' \\ \hline
&1&3&&1&&&&&&9&&&&z_3&&&z_3'\\ \hline
3&&1&3&&1&&&&&&9&&&z_4&&&z_4' \\ \hline
&3&&1&3&&1&&&&&&9&&z_5&&& z_5'\\ \hline
&&3&&1&3&&1&&&&&&9&z_6&&&z_6' \\ \hline
&&&3&&1&3&&1&&&&&&z_7&&&z_7' \\ \hline
&&&&3&&1&3&&1&&&&&z_8&*&&z_8' \\ \hline
&&&&&3&&1&3&&1&&&&z_9&&*&z_9' \\ \hline
9&&&&&&3&&1&3&&1&&&z_{10}&&&z_{10}' \\ \hline
&9&&&&&&3&&1&3&&1&&z_{11}&&&z_{11}' \\ \hline
&&9&&&&&&3&&1&3&&1&z_{12}&&&z_{12}' \\ \hline
&&&9&&&&&&3&&1&3&&z_{13}&&&z_{13}' \\ \hline
&&&&9&&&&&&3&&1&3&z_{14}&&&z_{14}' \\ \hline \hline
&&&&&9&&&&&&3&&1&\cellcolor{gray}&\cellcolor{gray}&\cellcolor{gray}& \cellcolor{gray}\\ \hline
1&&&&&&9&&&&&&3&&\cellcolor{gray}&\cellcolor{gray}&\cellcolor{gray}& \cellcolor{gray}\\ \hline
&1&&&&&&9&&&&&&3&\cellcolor{gray}&\cellcolor{gray}&\cellcolor{gray}& \cellcolor{gray}\\ \hline
\end{array}
\]
We then have:
\[
\phi_2(z_1')=z_{13} \quad \phi_2(z_3')=z_1 \quad \phi_2(z_9')=z_7.
\]
Note also that if we apply  the bishop's move function $3$ times to $z_j$, for $j \in \{1,7,13\}$, we obtain $z_{i+1}$, where the index $i$ is such that $\phi_2(z_i')=z_j$; we have shown this in the previous array with the cells marked with $*$, starting from $z_7$, and it can be seen that we arrive in $z_{10}'$, that is $z_{9+1}'$.

Hence we have that, for every odd integer $n\equiv 0\pmod 7$, $n\geq35$, there exists a solution to $P(A)$ where $A$ is cyclically $29$-diagonal array of order $n$.
\end{ex}

\begin{cor} \label{cor:case2}
Let $n\geq9$ and  $g$ be odd integers with $n$ divisible by $g$.
Let $A$ be a cyclically $k$-diagonal  $n \times n$ array, where $k=n-g+1$.
Then there exists a solution to $P(A)$.
\end{cor}
\begin{proof}
Note that, by the hypotheses, $g=\gcd(n,k-1)$.
Set $n = mg$ and $\ell = m-1$, then $k-1=\ell g$.  
Now we construct the permutation $\Theta$ of the intervals $I_1,\dotsc,I_m$ as defined in (3.2): since $-\ell \equiv 1 \pmod{m}$, we simply get $\Theta =(I_1,I_2,\dotsc,I_m)$. 
Finally set $E =( \bigcup_{i=1}^{\frac{m+1}{2}}I_i) \setminus [\frac{m+1}{2}-\frac{g-1}{2}, \frac{m+1}{2}]$. It is easy to see that $E$ satisfies Conditions (\ref{eq:cond1}) and (\ref{eq:cond2}), and moreover $|E| = \frac{k-1}{2} + \frac{g-1}{2}$. By Theorem \ref{prop:case2}, we conclude that there exists a solution to $P(A)$.
\end{proof}

\section{Conclusions}
The results presented in Section \ref{sec:new_sol} allows to prove that Conjecture \ref{conj:crazy_tour} holds for many infinite values of $n$ and $k$ as pointed out also in the following remarks.

\begin{rem} \label{rem:prime_times_2}
   From Corollary \ref{cor:case1} we can construct  infinite families of cyclically diagonal arrays that admit a solution to the Crazy Knight's Tour Problem. For example, if $g \equiv 1 \pmod{4}$ then there exists a solution to $P(A)$ for any cyclically $(1+2^q g)$-diagonal $n\times n$ array, where $q\geq 1 $ and $n \geq 1+2^q g$ odd.
\end{rem}

\begin{rem}
We can give a heuristic estimate on the number of cases solved by Theorem \ref{prop:case1}. Indeed, let $k$ be an odd integer and $g$ be an odd divisor of $k-1$, hence $k-1=2rg$ for some positive integer $r$. Let $s =\frac{g+1}{2}$, the condition $\gcd(g+1,k-1)=2$ becomes:
\[
\gcd(g+1,k-1)=2 \Leftrightarrow \gcd(2s,2r(2s-1))=2 \Leftrightarrow \gcd(s,r)=1.
\]
The correspondence  $g=2s-1$ and $k=2r(2s-1)+1$ is a bijection between positive integer pairs $(r,s)$ and admissible pairs $(k,g)$. Hence, from the well known  asymptotic probability $6/\pi^2$ that two randomly chosen numbers $r,s$  are coprime, the asymptotic proportion of pairs $(k,g)$ satisfying the condition of Theorem \ref{prop:case1} is $6/\pi^2\approx 60.79\%$. From each of these pairs we may then
choose any odd $n\equiv 0 \pmod{g}$ such that $g = \gcd(n,k-1)$, and find a solution to $P(A)$, where $A$ is a cyclically $k$-diagonal $n\times n$ array.
\end{rem}

From Theorem \ref{prop:case1} we may also completely solve the existence problem of solutions to the Crazy Knight's Tour Problem on cyclically $k$-diagonal arrays for some values of $k$, by directly checking the divisors of $k-1$. Indeed, if the conditions of Theorem \ref{prop:case1} are satisfied for every divisor $g$ of $k-1$, then a solution to the tour problem always exists.
From this consideration and some results of  \cite{CDP}, we deduce that for
\begin{equation} \label{eq:new_values_k}
 k \leq 245, \ k \not \in \{205,   211,   217,   221,   225,   229,   237,   241\}
\end{equation}
there exists a solution to the Crazy Knight's Tour Problem for every cyclically $k$-diagonal  array of order $n$, satisfying the necessary conditions $n\geq k \geq 3$ and $n,k$ odd.

We conclude this section showing how we can combine the results presented in Section \ref{sec:new_sol}
with Theorem \ref{prop:sol_ext} to get new solutions to $P(A)$.

\begin{prop} \label{prop:finale}
    Let $g\geq 3$ odd and let $ k\equiv 1 \pmod{g}$. There exists a solution to $P(A)$ for every  cyclically $k$-diagonal array  $A$ of order $n' \equiv g \pmod{k-1}$, $n'\geq k-1+g$.
\end{prop}
\begin{proof}
Let $g, k$ be as in the statement and let $n = k-1+g$. Since $g = \gcd(n,k-1)$, from Corollary \ref{cor:case2} there exists a solution to $P(A)$, where  $A$ is a   cyclically $k$-diagonal array  of order $n$. Since this solution has $C = (1,\dotsc, 1)$, from Theorem \ref{prop:sol_ext}  we may extend this solution to every cyclically $k$-diagonal array of order $n'=n+\lambda(k-1)$, $\lambda \geq 1$, hence the statement.
\end{proof}
Note that if in the previous result we moreover require that $k \equiv 1 \pmod{g(g+1)}$, then $\gcd(g+1,k-1)=g+1\neq 2$ and the constructed solution cannot be found by applying Theorem \ref{prop:case1}; this means that we actually have new solutions to the Crazy Knight's Tour Problem.
Indeed, in the following table we give some values of $k$  excluded from  (\ref{eq:new_values_k}) obtainable from Proposition \ref{prop:finale}, with their relative choices of $g$ and $n$.
\[
\begin{array}{c|c|l}
    k & g  & n\\ \hline
    205 & 3 & 3 \pmod{204} \\
    211 & 5& 5 \pmod{210}\\
    217 & 3& 3 \pmod{216}\\
    225 & 7& 7 \pmod{224} \\
    229 & 3& 3 \pmod{228} \\
    241 & 3,5,15& 3,5,15 \pmod{240} \\
\end{array}
\]

\section*{Acknowledgements}
 The authors would like to thank Simone Costa for the interesting discussion on the topic.

\end{document}